\documentclass[12pt]{amsart}%
\usepackage{color}
\usepackage{amsmath}
\usepackage{amssymb}
\usepackage{amsfonts}
\usepackage[latin1]{inputenc}
\usepackage{graphicx}%
\usepackage{cite}
\providecommand{\U}[1]{\protect\rule{.1in}{.1in}}

\DeclareMathSymbol{\subsetneqq}{\mathbin}{AMSb}{36}

\theoremstyle{plain}
\numberwithin{equation}{section}
\newtheorem{theorem}{Theorem}[section]
\newtheorem{corollary}{Corollary}[section]
\newtheorem{lemma}{Lemma}[section]
\newtheorem{proposition}{Proposition}[section]

\newtheorem{remark}{Remark}[section]

\begin{document}
\title[On the convergence of the continuous gradient projection method]
{On the convergence of the continuous gradient projection method}%
\author{Ramzi May}%
\address{Mathematics and Statistics Department, College of Science, King Faisal University, P.O. 380 Ahsaa 31982, Kingdom of Saudi Arabia}
\email{rmay@kfu.edu.sa}
\vskip 0.2cm
\date{28 October, 2018}
\dedicatory{}
\begin{abstract}
In a Hilbert setting $\mathcal{H}$, we study the weak
and the strong convergence properties of the trajectories $x(t)$ of the
continuous gradient projection dynamical system
\[
x^{\prime}(t)+x(t)=P_{Q}(x(t)-\lambda(t)\nabla f(x(t)),~t\geq0,%
\]
where $Q$ is a closed, non-empty and convex subset of
$\mathcal{H},$ the function $f:\mathcal{H}\rightarrow\mathbb{R}$ \ is
regular and convex, the operator $P_{Q}:\mathcal{H}\rightarrow Q$ is the
projection onto $Q,$ and $\lambda:[0,+\infty\lbrack\rightarrow]0,+\infty\lbrack$
is a an absolutely continuous function. We prove the weak convergence of the
trajectories $x(t)$ to a minimizer of $f$ over $Q$, if one exists, under some mild
hypothesis on the function $\lambda(.).$ We also study the strong
convergence and the decay rate to equilibrium of the trajectories
under a global Holderian error bound assumption on the objective function $f$ of the form:
\[
\left(  f(x)-f_{Q}^{\ast}\right)  ^{\theta}\geq \kappa~ dist(x,\arg\min
_{Q}f)~\forall x\in Q
\]
where $\kappa>0$ and $ \theta\in]0,\frac{1}{2}]$ are absolute constants, $f_Q^*=min_{Q}f$, and $\arg\min
_{Q}f$ is the set of the minimizers of $f$ over $Q$.
\end{abstract}
\maketitle

\section{Introduction and statement of the results}

Let$\ \mathcal{H}$ be a real Hilbert space endowed with the inner product
$\langle.,.\rangle$ and the associated norm $\left\Vert .\right\Vert .$
Throughout this paper, $Q$ is a closed, convex and non-empty subset of
$\mathcal{H},$ $\lambda:[0,+\infty\lbrack\rightarrow]0,+\infty\lbrack$ is an
absolutely continuous function which belongs to the space $W^{1,1}_{loc}([0,+\infty[)$ and $f:\mathcal{H}\rightarrow\mathbb{R}$ \ is a
$C^{1}$ convex function such that its gradient function $\nabla f$ is locally Lipschitz on $ \mathcal{H}$. We assume moreover that $f$ achieves its minimum over $Q$ at
least in one point, which means the set of the minimizers of $f$ over $Q$%
\begin{equation}
\arg\min_{Q}f:=\{x\in H:f(x)=f_{Q}^{\ast}:=\min_{Q}f\}\label{e2}%
\end{equation}
is non-empty. We consider the
continuous gradient projection dynamical system%
\begin{equation}
\left\{
\begin{array}
[c]{c}%
x^{\prime}(t)+x(t)=P_{Q}\left(  x(t)-\lambda(t)\nabla f(x(t))\right)
,~t\geq0,\\
x(0)=x_{0},
\end{array}
\right.  \tag{CGP}%
\end{equation}
where $x_{0}$, the initial data, is a given element of $Q$ and
$P_{Q}:\mathcal{H}\rightarrow Q$ is the projection onto $Q$. The system (CGP) is the continuous version of the discrete gradient
projection algorithm
\begin{equation}
\left\{
\begin{array}{c}
x_{0}\in Q\text{ (given)} \\
x_{k+1}=P_{Q}(x_{k}-\alpha _{k}\nabla f(x_{k})),~k\in \mathbb{N},%
\end{array}%
\right.   \tag{DGP}
\end{equation}%
suggested by Goldestein, Levitin and Polyak \cite{Gold64, Lev65}, where $%
(\alpha _{k})_{k}$ is a given sequence of non negative real numbers. For the
sudy of the convergence properties of the algorithm (DGP), we refer the
readers for instance to the references \cite{Gold64, Lev65, Xu, Be}.
In this paper, we limit our selfs to the study of the convergence properties
of the system (CGP). We recall that Antipin \cite{Ant} and Bolte \cite{Bo} studied the system (CGP) in the case where the
function $\lambda(t)\equiv\lambda$ is a non negative constant. They proved that, for every
initial data $x_{0}\in Q$, the Cauchy problem (CGP) has a unique global
solution $x\in C^{1}([0,+\infty\lbrack,\mathcal{H})$ that converges weakly as
$t\rightarrow+\infty$ to some minimizer of $f$ over $Q$ and there exists a
constant $C>0$ such that for every $\ t>0$, $f(x(t))-f_{Q}^{\ast}%
\leq\frac{C}{t}$. In the following theorem, we extend this convergence
properties of the system (CGP) to a large class of functions $\lambda.$

\begin{theorem}
\label{Th1}
Let $x_{0}\in Q$. Then the
Cauchy problem (CGP) has a unique global solution $x\in C^{1}([0,+\infty
\lbrack,\mathcal{H})$ such that
$x(t)\in Q$ for all $t\geq0$ . Moreover, if in addition the function $\lambda$ satisfies the two following conditions%
\begin{equation}
\int_{0}^{+\infty}\lambda(t)dt=+\infty,\label{e3}%
\end{equation}
and%
\begin{equation}
\int_{0}^{+\infty} | \lambda^{\prime}(t)|dt<+\infty
,\label{e4}%
\end{equation}
then $x(t)$ converges weakly as $t\rightarrow+\infty$ to some
$x^{\ast}$ in $\arg\min_{Q}f$ and
\begin{equation}
\lim_{t\rightarrow+\infty}\Gamma(t)\left(  f(x(t))-f_{Q}^{\ast}\right)
=0,\label{e5}%
\end{equation}
where
\begin{equation}
\Gamma(t)=\int_{0}^{t}\lambda(s)ds.\label{e6}%
\end{equation}

\end{theorem}

By adapting the arguments of Bruck \cite[Theorem 5]{Br2} and Brezis \cite[Theorem 3.13]{Bri}, we
prove the strong convergence of the trajectories $x(t)$ of (CGP) under some additional
geometrical assumptions on $Q$ and $f$.

\begin{theorem}
\label{Th2}
Under the hypothesis (\ref{e3}) and (\ref{e4}), additionally assume that one of the
two following assumptions holds:
\begin{enumerate}
\item The set $Q$ is symmetric with respect to the origin and the function
$f$ is even on $Q$ (i.e. for all $x\in Q,-x\in Q$ and $f(-x)=f(x)$).

\item The set $\arg\min_{Q}f$ has an interior point $x_{0}^{\ast}$ in $\mathcal{H}$.
\end{enumerate}
Then for any $x_{0}\in Q$ the solution $x\in C^{1}([0,+\infty\lbrack
,\mathcal{H})\ $ of (CGP) converges strongly as $t\rightarrow+\infty$ to some
$x^{\ast}$ in $\arg\min_{Q}f$ .
\end{theorem}
Motivated by the works of Dunn \cite{Dun} and Necoara, Nestrov and Glineur\cite{Nec} in their studies of the convergence of the (DGP) algorithm, we investigate in the
second part of this paper the strong convergence and the decay rate to
equilibrium of the trajectories of the continuous gradient projection system
(CGP) under a global Holderian error bound condition on the objective function
$f$.  Precisely, let us assume in the sequel of this introduction, that $f$
satisfies the following (GHEB) hypothesis: There exists $\kappa>0,$ and
$\theta\in]0,1]$ such
\begin{equation}
\left(  f(x)-f_{Q}^{\ast}\right)  ^{\theta}\geq\kappa~dist(x,\arg\min
_{Q}f)~\forall x\in Q.\label{e7}%
\end{equation}

Before setting our main result on the strong convergence of the trajectories of (CGP) under the hypothesis (GHEB), let us first set a simple proposition which provides a class of convex functions
satisfying this hypothesis with $\theta\in]0,\frac{1}{2}].$

\begin{proposition}
\label{Pr1}
Let $g:\mathcal{H}\rightarrow\mathbb{R}$ a non negative strongly convex
function of class $C^{1}$ and $\theta\in]0,\frac{1}{2}].$ Then there exists
$\kappa>0$ such that the convex function $h(x):=\left(  g(x)\right)
^{\frac{1}{2\theta}}$ satisfies
\begin{equation}
\left(  h(x)-h_{Q}^{\ast}\right)  ^{\theta}\geq\kappa~dist(x,\arg\min
_{Q}h)~\forall x\in Q.%
\end{equation}
\end{proposition}
Now we set the main result of this paper.
\begin{theorem}
\label{Th3}
Assume that $\lambda^{\prime}(t)\leq0$ for almost every $t\geq0$ and there
exist $\kappa>0$ and $\theta\in]0,\frac{1}{2}]$ such that $f$ satisfies
(\ref{e7}). Let $x_{0}\in Q$ and $x$ be the unique solution of (CGP) in
$C^{1}([0,+\infty\lbrack,\mathcal{H}).$

\begin{enumerate}
\item If $\theta\in]0,\frac{1}{2}[$ and
\begin{equation}
\int_{0}^{+\infty}\frac{1}{1+t}\left(  \Gamma(t)\right)  ^{-\frac{\theta
}{1-2\theta}}dt<\infty,\label{H1}%
\end{equation}
then $x(t)$ converges strongly as $t\rightarrow+\infty$ to some $x_{\infty}$
in $\arg\min_{Q}f$ and%
\begin{align}
\left\Vert x(t)-x_{\infty}\right\Vert  & =O\left(  \int_{t}^{+\infty}\frac
{1}{s}\left(  \Gamma(s)\right)  ^{-\frac{\theta}{1-2\theta}}ds\right)
\label{H3}\\
f(x(t))-f_{Q}^{\ast}  & =O\left(  \frac{\left(  \Gamma(t)\right)
^{-\frac{2\theta}{1-2\theta}}}{t\lambda(t)}\right)  \label{H4}%
\end{align}

\item If $\theta=\frac{1}{2}$ and for every $c>0,$
\[
\int_{0}^{+\infty}\frac{1}{1+t}e^{-c\Gamma(t)}dt<\infty,
\]
then $x(t)$ converges strongly as $t\rightarrow+\infty$ to some $x_{\infty}$
in $\arg\min_{Q}f$ and there exists $\mu>0$ such that%
\begin{align}
\left\Vert x(t)-x_{\infty}\right\Vert  & =O\left(  e^{-\mu\Gamma(t)}\right)
\label{N1}\\
f(x(t))-f_{Q}^{\ast}  & =O\left(  e^{-\mu\Gamma(t)}\right)  .\label{N2}%
\end{align}

\end{enumerate}
\end{theorem}

An immediate consequence of this theorem is the following important result on
the asymptotic behavior of the trajectories of (CGP) when the function
$\lambda$ behaviors for $t$ large enough like $\frac{K}{t^{\alpha}}$ for some
$K>0$ and $0<\alpha<1$.

\begin{corollary}
\label{Co1}
Assume that $\lambda(t)=\frac{K}{(1+t)^{\alpha}}$ with $K>0$ and $0<\alpha<1$
are absolute constants and there exists $\kappa>0$ and $\theta\in]0,\frac
{1}{2}]$ such that $f$ satisfies (\ref{e7}). Then for every $x_{0}\in Q,$ the
unique solution $x$ of (CGP) in $C^{1}([0,+\infty\lbrack,\mathcal{H})$
converges strongly as $t\rightarrow+\infty$ to some $x_{\infty}$ in $\arg
\min_{Q}f$ and
\begin{align}
\left\Vert x(t)-x_{\infty}\right\Vert  &  =O\left(  t^{-\frac{\left(
1-\alpha\right)  \theta}{1-2\theta}}\right)  ,\label{H5}\\
f(x(t))-f_{Q}^{\ast} &  =O\left(  t^{-\frac{\left(  1-\alpha\right)
}{1-2\theta}}\right)  ,\label{H6}%
\end{align}
if $\theta\in]0,\frac{1}{2}[,$ and
\begin{align}
\left\Vert x(t)-x_{\infty}\right\Vert  & =O\left(  e^{-\delta\Gamma
(t)}\right)  \label{N3}\\
f(x(t))-f_{Q}^{\ast}  & =O\left(  e^{-\delta\Gamma(t)}\right)  \label{N4}%
\end{align}
for some constant $\delta>0$ if $\theta=\frac{1}{2}.$
\end{corollary}
In the unconstrained case where $Q=\mathcal{H}$ and (CGP) is the gradient
system%
\begin{equation}
\left\{
\begin{array}{c}
x^{\prime }(t)+\lambda (t)\nabla f(x(t))=0,~t\geq 0, \\
x(0)=x_{0},%
\end{array}%
\right.   \tag{GS}
\end{equation}%
we can prove a more precise result than Theorem \ref{Th3} under the sole
hypothesis (\ref{e3}) on the function $\lambda $ and a weaker assumption on the
function $f$ than the (GHEB) hypothesis. In fact, according to
\cite[Theorem 30]{BDLM} and \cite[Theorem 5]{BNPS}, if the convex function $f$ satisfies the
hypothesis (\ref{e7}), then it satisfies the following global Lojasiewicz inequality
\begin{equation}
\varphi ^{\prime }(f(x)-f_{Q}^{\ast })\left\Vert \nabla f(x)\right\Vert \geq
1  \label{Ch}
\end{equation}%
for all $x\in Q$ such that $f(x)>f_{Q}^{\ast },$ where $\varphi $ is the
desingularizing function defined on $[0,+\infty \lbrack $ by $\varphi (s)=%
\frac{s^{\theta }}{\kappa \theta }.$ Therefore, by adapting the approach of
Haraux and Jendoubi \cite[Chapter 9]{HJ} and Chill and Fioranza
\cite[Theorem 2.7]{CF}, we can prove the following precise result on the strong convergence and
the decay rate to equilibrium of the trajectories of (GS) under a local
version of the inequality (\ref{Ch}).
\begin{theorem}
\label{Th4}
Additionally to (\ref{e3}), assume that there exists $x^{\ast }$ a global minimizer of
$f$ over $\mathcal{H}$ such that
\begin{equation}
\varphi ^{\prime }(f(x)-f^{\ast })\left\Vert \nabla f(x)\right\Vert \geq
1~\forall x\in U(x^{\ast },r^{\ast }),  \label{Ch1}
\end{equation}%
where $r^{\ast }\in ]0,+\infty ]$, $f^{\ast }=\min_{\mathcal{H}}f,~U(x^{\ast },r^{\ast
})=\{x\in \mathcal{H}:\left\Vert x-x^{\ast }\right\Vert <r^{\ast }$ and $%
f(x)>f^{\ast }\},$ and $\varphi (s)=\frac{s^{\theta }}{\kappa \theta }$ for
every $s\geq 0$ where $\kappa >0$ and $\theta \in ]0,\frac{1}{2}]$ are some absolute constants. Then there exists $r_{0}\in ]0,r^{\ast }]$ such that for
every $x_{0}\in B_{\mathcal{H}}(x^{\ast },r_{0})$, the unique global solution
$x\in C^{1}([0,+\infty \lbrack ,H)$ of the gradient system (GS) converges
strongly to some global minimizer $x_{\infty }$\ of $f$ over $\mathcal{H}$.
Moreover,

\begin{enumerate}
\item If $\theta \in ]\frac{1}{2},1]$ then $f(x_{0})=f^{\ast }$ and $%
x(t)=x_{0}$ for every $t\geq 0.$

\item If $\theta =\frac{1}{2}$ then there exists $\delta >0$ such that
\begin{equation}
\left\Vert x(t)-x_{\infty }\right\Vert =O(e^{-\delta \Gamma (t)}),%
\end{equation}%
and
\begin{equation}
f(x(t))-f^{\ast}=O(e^{-\delta \Gamma (t)}).
\end{equation}%
\item If $\theta \in ]0,\frac{1}{2}[$ then
\begin{equation}
\left\Vert x(t)-x_{\infty
}\right\Vert =O\left( \left( \Gamma (t)\right) ^{-\frac{\theta }{1-2\theta }%
}\right),
\end{equation}
and
\begin{equation}
f(x(t))-f^{\ast}=O\left( \left( \Gamma (t)\right) ^{-\frac{1 }{1-2\theta }}\right).
\end{equation}
\end{enumerate}
\end{theorem}
\begin{remark}
In the case where $\theta\in]0,\frac{1}{2}]$ and $\lambda(t)=\frac{K}{(1+t)^{\alpha}}$ for some constants $K>0$ and $0<\alpha<1$, Corollary \ref{Co1} and the Theorem \ref{Th4} give the same estimates on the decay rates of the trajectories of (CGP) and (GS). This let us ask wether the precise Theorem \ref{Th4} holds true for the general system (CGP).
\end{remark}
\begin{remark}
Recently, Frankel, Garrigos, and Peypouquet in \cite[Theorem 3.4]{FGP} have proved an analogues result to Theorem \ref{Th4} for the discrete gradient projection algorithm (DGP). This let us hope that Theorem \ref{Th4} can be extended to cover the continuous system (CGP).
\end{remark}
The rest of the paper is organized as follows: In the next section we prove Theorem \ref{Th1} on the weak convergence of (CGP). The third section is devoted to the proofs of Theorem \ref{Th2} and Theorem \ref{Th3} on the strong convergence of the trajectories of (CGP) under some geometrical hypothesis on $Q$ and the objective function$f$. In the fourth section, we provide short proofs of Proposition \ref{Pr1} and Corollary \ref{Co1} and a detailed proof of the main Theorem \ref{Th3}. The last section is devoted to the proof of Theorem \ref{Th4} on the convergence properties of the particular gradient system (GS) under the local Lojasiewicz inequality (\ref{Ch1}).
\section{General weak convergence result of (CGP)}
This section is devoted to the proof of Theorem \ref{Th1}. The proof is essentially based on two
elementary results.
The first one is  a version of \cite[Lemma 2.2]{Al}).
\begin{lemma}
\label{Le1}
Let $u:[0,+\infty \lbrack \rightarrow \mathbb{R}$ be an absolutely
continuous function bounded from below. If $\left[ u^{\prime }\right] ^{+}:=\max{(u^{\prime },0)}$
the positive part of the derivative $u^{\prime },$ belongs to $%
L^{1}([0,+\infty \lbrack ;\mathbb{R})$ then $%
u^{\prime}\in L^{1}([0,+\infty \lbrack ;\mathbb{R})$ and $u(t)$ converges as $%
t\rightarrow +\infty .$
\end{lemma}
\begin{proof}
Let $\left[ u^{\prime }\right] ^{-}:=\max{(-u^{\prime },0)}$. Since $u^{\prime }=\left[ u^{\prime }\right]^{+}-\left[ u^{\prime }\right]^{-}$, then for every $T>0$
\[\int_0^T\left[ u^{\prime }\right]^{-}(t)dt\leq\int_0^{+\infty}\left[ u^{\prime }\right]^{+}(t)dt+u(0)-\inf_{t\geq0}u(t).
\]%
Therefore $\left[ u^{\prime }\right] ^{-}$, and by consequence $u^{\prime}$, belong to the space $L^{1}([0,+\infty \lbrack ;\mathbb{R})$, which implies that $u(t)$ converges as $%
t\rightarrow +\infty .$
\end{proof}
The second key result is the continuous version of the classical Opial's
lemma \cite{Op} (see \cite{AGR} for a simple and clear proof)
\begin{lemma}[Opial's lemma]Let $x:[t_{0},+\infty)\rightarrow\mathcal{H}.$ Assume
that there exists a non-empty subset $S$ of $\mathcal{H}$ such that:
\begin{enumerate}
\item[i)] if $t_{n}\rightarrow+\infty$ and $x(t_{n})\rightharpoonup x$ weakly
in $\mathcal{H}$ , then $x\in S$,
\item[ii)] for every $z\in S,$ $\displaystyle\lim_{t\rightarrow+\infty}\left\Vert
x(t)-z\right\Vert $ exists.
\end{enumerate}
\noindent Then there exists $z_{\infty}\in S$ such that $x(t)\rightharpoonup
z_{\infty}$ weakly in $\mathcal{H}$ as $t\rightarrow+\infty.$
\end{lemma}
Now we are ready to prove Theorem \ref{Th1}.
\begin{proof}
Let $x_{0}\in Q.$ Since $P_{Q}$ is non-expansive mapping and $\nabla f$
is locally Lipschitz, then according to the Cauchy-Lipschitz theorem, the system
(CGP) has a unique maximal solution $x\in C^{1}([0,T^{\ast }[;\mathcal{H}).$
By proceeding exactly as in the beginning of the proof of \cite[Theorem 2.1]{Bo}, we deduce that $x(t)\in Q$ for every $t\in \lbrack 0,T^{\ast }[.$ On
the other hand, from the characterization of the projection operator $P_{Q},$
\begin{equation}
\langle x^{\prime }(t)+x(t)-w,x^{\prime }(t)+\lambda (t)\nabla
f(x(t))\rangle \leq 0~\forall w\in Q.  \label{S1}
\end{equation}%
Hence by letting $w=x(t),$ we get%
\begin{equation}
\left\Vert x^{\prime }(t)\right\Vert ^{2}+\lambda (t)\left(
f(x(t))-f_{Q}^{\ast }\right) ^{\prime }\leq 0,  \label{S2}
\end{equation}%
which implies, in particular, that the non negative function $%
f(x(t))-f_{Q}^{\ast }$ is non increasing on the interval $[0,T^{\ast }[.$
Therefore from (\ref{S2}), we infer that
\begin{equation}
\left\Vert x^{\prime }(t)\right\Vert ^{2}+\left( \lambda (t)\left(
f(x(t))-f_{Q}^{\ast }\right) \right) ^{\prime }\leq |\lambda ^{\prime }%
(t)|\left( f(x_{0})-f_{Q}^{\ast }\right) .  \label{S3}
\end{equation}%
Integrating this inequality on $[0,T^{\ast }[,$ we obtain
\begin{eqnarray*}
\int_{0}^{T^{\ast }}\left\Vert x^{\prime }(t)\right\Vert ^{2} &\leq &\left(
f(x_{0})-f_{Q}^{\ast }\right) \left( \lambda (0)+\int_{0}^{T^{\ast} }
|\lambda ^{\prime }(t)|dt\right)
\end{eqnarray*}%
From the last inequality, we deduce by a standard argument that $T^{\ast
}=+\infty .$ Indeed, let us argue by contradiction and suppose that $T^{\ast
}<+\infty .$ Then the Cauchy-Schawrz inequality combined with previous
inequality yields that  $\int_{0}^{T^{\ast }}\left\Vert x^{\prime
}(t)\right\Vert dt<+\infty $ which implies that $\lim_{t\rightarrow T^{\ast
}}x(t)=\tilde{x}$ exists. Then by applying again the Cauchy Lipschitz
theorem to the system%
\[
\left\{
\begin{array}{c}
x^{\prime }(t)+x(t)=P_{Q}(x(t)-\lambda (t)\nabla f(x(t))),~t\geq T^{\ast }
\\
x(T^{\ast })=\tilde{x},%
\end{array}%
\right.
\]%
we deduce that we can extend the solution $x(.)$ of (CGP) on an interval
strictly larger than $[0,T^{\ast }[$, which contradicts the definition of $%
T^{\ast }.$ Therefore $T^{\ast }=+\infty .$
Let us now apply Opial's lemma to prove the weak convergence of the trajectory $x(t)$ as $t\rightarrow +\infty$ under the additionally assumptions (\ref{e3}) and (\ref{e4}) on the function $\lambda$. Let $z$ be an arbitrary element of $\arg \min_{Q}f$ and define the function $%
\varphi _{z}$ on $[0,+\infty \lbrack $ by
\begin{equation}
\varphi _{z}(t)=\frac{1}{2}\left\Vert x(t)-z\right\Vert ^{2}.  \label{S4}
\end{equation}%
Going back to (\ref{S1}) and let $w=z,$ we get after some trivial
simplification
\begin{equation}
\varphi _{z}^{\prime }(t)+\left(
\lambda (t)\left( f(x(t))-f_{Q}^{\ast }\right) \right) ^{\prime }+\lambda
(t)\langle \nabla f(x(t)),x(t)-z\rangle \leq \ |\lambda ^{\prime }(t)|%
\left( f(x_{0})-f_{Q}^{\ast }\right). \label{HH}
\end{equation}%
Using now the convexity inequality%
\[
f(z)\geq f(x(t))+\langle \nabla f(x(t)),z-x(t)\rangle
\]%
and the fact that $f(z)=f_{Q}^{\ast },$ we get the differential inequality%
\begin{equation}
\xi _{z}^{\prime }(t)+\lambda (t)\left( f(x(t))-f_{Q}^{\ast }\right) \leq
|\lambda ^{\prime }(t)|\left( f(x_{0})-f_{Q}^{\ast }\right)
,  \label{S5}
\end{equation}%
where%
\[
\xi _{z}(t):=\varphi _{z}(t)+\lambda (t)\left( f(x(t))-f_{Q}^{\ast }\right) .
\]%
From (\ref{S5}), $\left[ \xi _{z}^{\prime }\right]^+ \in L^{1}([0,+\infty
\lbrack ;\mathbb{R}),$ therefore according to Lemma \ref{Le1}, $\xi _{z}(t)$
converges as $t\rightarrow +\infty .$ Again, from Lemma \ref{Le1} and the
inequality (\ref{S3}), we deduce that $\ \lambda (t)\left(
f(x(t))-f_{Q}^{\ast }\right) $ converges as $t\rightarrow +\infty .$ Hence, $%
\lim_{t\rightarrow +\infty }\varphi _{z}(t)$ exists. On the other hand, by
integrating (\ref{S5}) on $[0,T]$ and letting $T$ goes to $+\infty ,$ we
infer that%
\begin{equation}
\int_{0}^{+\infty }\lambda (t)\left( f(x(t))-f_{Q}^{\ast }\right) dt<+\infty
.  \label{S6}
\end{equation}%
But the non negative function $f(x(t))-f_{Q}^{\ast }$ is decreasing, hence
it converges as $t$ goes to $+\infty $ to some real $l_{\infty }\geq 0$
which, according to (\ref{S6}) and the assumption $\int_{0}^{+\infty
}\lambda (t)dt=+\infty ,$ must be equal to $0.$ Therefore, from the facts
that $\{x(t):t\geq 0\}\subset Q,$ the set $Q$ is weakly closed (since it is
convex and closed), and the weak lower semi-continuity of the convex
function $f,$ we deduce that if $t_{n}\rightarrow +\infty $ and $x(t_{n})$
converges weakly to some $x^{\ast }$, then $x^{\ast }$ belongs to $Q$ and
satisfies $f(x^{\ast })\leq f_{Q}^{\ast },$ which means that $x^{\ast }$ is
in $\arg \min_{Q}f.$ Thus, from Opial's lemma, $x(t)$ converges weakly as $%
t\rightarrow +\infty $ to some minimizer of $f$ over $Q.$ To end the proof
of Theorem \ref{Th1}, it remains to prove (\ref{e5}). Let $\varepsilon >0.$ From (\ref{S6}),
there exists $t_{0}>0$ such that for every $t\geq t_{0},$%
\[
\int_{t_{0}}^{t}\lambda (s)\left( f(x(s))-f_{Q}^{\ast }\right) ds\leq
\varepsilon .
\]%
Using now the fact that the function $f(x(t))-f_{Q}^{\ast }$ is decreasing,
we get
\[
\Gamma (t)\left( f(x(t))-f_{Q}^{\ast }\right) \leq \varepsilon +\Gamma
(t_{0})\left( f(x(t))-f_{Q}^{\ast }\right) .
\]%
Letting $t\rightarrow +\infty $ and using the fact that $\lim_{t\rightarrow
0}f(x(t))=f_{Q}^{\ast },$ we obtain
\[
\lim \sup_{t\rightarrow +\infty }\Gamma (t)\left( f(x(t))-f_{Q}^{\ast
}\right) \leq \varepsilon ,
\]%
which implies the required estimate (\ref{e5}).
\end{proof}
\section{Strong convergence of the system (CGP) under some geometrical
properties of the objective function}
In this section we prove Theorem\ref{Th2}. We will use some results established in
the proof of Theorem \ref{Th1}.
\begin{proof}
Let $x_{0}\in Q$ and $x(.)$ be the global solution of the system (CGP). Let
us prove the strong convergence of the trajectory $x(.)$ under the first assumption of Theorem \ref{Th2}. The proof is inspired by the proof of \cite[Theorem 5]{Br2}. Let $%
0<t_{1}<t_{2},$ and define on $[t_{1},t_{2}]$ the function $g(s):=\left\Vert
x(s)\right\Vert ^{2}-\left\Vert x(t_{2})\right\Vert ^{2}-\frac{1}{2}%
\left\Vert x(s)-x(t_{2})\right\Vert ^{2}.$ It is clear that%
\[
g^{\prime }(s)=\langle x^{\prime }(s),x(s)+x(t_{2})\rangle .
\]%
From Theorem \ref{Th1} and the the symmetry of $Q$ with respect of $0$, we have $-x(t_{2})\in Q.$ Therefore, by letting $w=-x(t_{2})$ in the inequality
(\ref{S1}), we get
\[
\left\Vert x^{\prime }(s)\right\Vert ^{2}+g^{\prime }(s)+\lambda (s)\left(
f(x(s))-f_{Q}^{\ast }\right) ^{\prime }+\lambda (s)\langle \nabla
f(x(s),x(s)+x(t_{2})\rangle \leq 0.
\]%
Hence by using the fact that $f$ is even, the convex inequality%
\[
f(-x(t_{2}))\geq f(x(s))+\langle \nabla f(x(s),-x(t_{2})-x(s)\rangle
\]%
and the fact that the function $f(x(t))-f_{Q}^{\ast }$ is decreasing on $%
[0,+\infty \lbrack $ (see the proof of Theorem \ref{Th1}), we get for every $s$ in $[t_{1},t_{2}]$
\begin{eqnarray*}
g^{\prime }(s)+\left( \lambda (s)\left( f(x(s))-f_{Q}^{\ast }\right) \right)
^{\prime } &\leq &\lambda (s)\left( f(x(t_{2})-f(x(s)\right) +|\lambda
^{\prime }(s)|\left( f(x(s)-f_{Q}^{\ast }\right)  \\
&\leq &|\lambda ^{\prime }(s)|\left( f(x_{0})-f_{Q}^{\ast
}\right) .
\end{eqnarray*}%
Integrating this inequality on $[t_{1},t_{2}],$ we obtain, after some simple
simplifications,%
\begin{equation}
\frac{1}{2}\left\Vert x(t_{1})-x(t_{2})\right\Vert ^{2}\leq \lambda
(t_{1})(f(x(t_{1}))-f_{Q}^{\ast}) + M_{0}\int_{t_{1}}^{t_{2}}|\lambda ^{\prime }(s)|ds+\left\Vert x(t_{1})\right\Vert ^{2}-\left\Vert
x(t_{2})\right\Vert ^{2},  \label{d}
\end{equation}%
where $M_{0}:=\left( f(x_{0})-f_{Q}^{\ast }\right) .$ Recall that in the
proof of Theorem \ref{Th1}, we have established that $%
\lim_{t\rightarrow +\infty }\lambda (t)\left( f(x(t))-f_{Q}^{\ast }\right)=0 $ (indeed we have proved that $\lim_{t\rightarrow +\infty }f(x(t))-f_{Q}^{\ast }=0$ and we know that the function $\lambda$ is bounded since its derivative belongs to $L^1([0,+\infty[)$
and that for every $z$ in $\arg \min_{Q}f,$ $\lim_{t\rightarrow
+}\left\Vert x(t)-z\right\Vert ^{2}$ exists. But from the hypothesis on $Q$
and $f,$ we have $0\in \arg \min_{Q}f$, then $\lim_{t\rightarrow
+}\left\Vert x(t)\right\Vert ^{2}$ exists. Hence  by letting $t_{1}$ and $t_{2}$ go $%
+\infty $ in the inequality (\ref{d}), we conclude that
\[
\left\Vert x(t_{1})-x(t_{2})\right\Vert \rightarrow 0\text{ as }%
t_{1},t_{2}\rightarrow +\infty ,
\]%
which implies that $x(t)$ converges strongly to some $x^{\ast}$ in $\mathcal{H}$ as $%
t\rightarrow +\infty .$ From Theorem \ref{Th1}, $x^{\ast}$ belongs to $\arg \min_{Q}f$.
Let us now prove the strong convergence of the trajectory $x(.)$ under the
second assumption of Theorem \ref{Th2}. The proof is
inspired by \cite[Theorem 3.13]{Bri}. By assumption, there $x_{0}^{\ast}$ in $\arg \min_{Q}f$ and $%
r^{\ast }>0$ such that for every $x$ in the ball $B_{\mathcal{H}}(x_{0}^{\ast},r^{\ast }),$ $f(x)=f_{Q}^{\ast }$ which implies that $\nabla f(x)=0.$
Hence, from the positivity of the operator $\nabla f,$ for every $y\in
\mathcal{H}$ and $v\in B_{\mathcal{H}}(0,1)$ we have
\[
\langle \nabla f(y),y-x_{0}^{\ast}-r^{\ast }v\rangle \geq 0,
\]%
which implies that
\begin{equation}
\left\Vert \nabla f(y)\right\Vert =\sup_{\left\Vert v\right\Vert <1}\langle
\nabla f(y),v\rangle \leq \frac{1}{r^{\ast }}\langle \nabla f(y),y-x^{\ast
}\rangle .  \label{d2}
\end{equation}%
Using now the facts that $x$ is the solution on (CGP), $x(t)\in Q$ for every
$t\geq 0,$ and the fact $P_{Q}$ is a non-expansive mapping, we deduce from (\ref{d2}) that
\begin{eqnarray*}
\left\Vert x^{\prime }(t)\right\Vert &=& \left\Vert P_Q(x(t)-\lambda(t)\nabla f(x(t)))-x(t)\right\Vert \\
&\leq &\lambda (t)\left\Vert \nabla
f(x(t))\right\Vert  \\
&\leq &\frac{1}{r^{\ast }}\lambda (t)\langle \nabla f(x(t)),x(t)-x_{0}^{\ast}\rangle .
\end{eqnarray*}%
Hence, from the inequality (\ref{HH}) with $z=x_{0}^{\ast},$ we deduce
\[
\left\Vert x^{\prime }(t)\right\Vert \leq \frac{1}{r^{\ast }}\left( |
\lambda ^{\prime }(t)|\left( f(x_{0})-f_{Q}^{\ast }\right)
-\varphi _{x_{0}^{\ast}}^{\prime }(t)-\left( \lambda (t)\left(
f(x(t))-f_{Q}^{\ast }\right) \right) ^{\prime }\right) .
\]%
Integrating this inequality, we infer that for every $t\geq 0$%
\[
\int_{0}^{t}\left\Vert x^{\prime }(s)\right\Vert ds\leq \frac{1}{r^{\ast }}%
\left( \left( f(x_{0})-f_{Q}^{\ast }\right) \int_{0}^{+\infty }|
\lambda ^{\prime }(s)|ds+\varphi _{x_{0}^{\ast}}(0)+\left( \lambda
(0)\left( f(x_{0})-f_{Q}^{\ast }\right) \right) \right) ,
\]%
which implies that $x^{\prime }$ belongs to the space $L^{1}([0,+\infty
\lbrack ;\mathcal{H}).$ Thus $x(t)$ converges strongly in $\mathcal{H}$ as $%
t\rightarrow +\infty .$ This completes the proof of Theorem \ref{Th2}.
\end{proof}
\section{Strong convergence of the system (CGP) under the (GHEB) hypothesis}
In this section, we prove Proposition \ref{Pr1}, Theorem \ref{Th3} and Corollary \ref{Co1}.
\par Let us first give the proof of Proposition \ref{Pr1}.
\begin{proof}
Since $g$ is continuous and strongly convex and $Q$ is a non-empty closed
convex subset of $\mathcal{H},$ the function $g$ has a unique minimizer $%
x_{Q}^{\ast }$ over $Q$. Then $\arg \min h=\{x_{Q}^{\ast }\}$ and $%
h_{Q}^{\ast }=\left( g(x_{Q}^{\ast })\right) ^{\frac{1}{2\theta }}.$
Moreover, since $g$ is $C^{1}$ and strongly convex, there exists a constant $%
m>0$ such that for every $x\in Q$%
\begin{eqnarray*}
g(x) &\geq &g(x_{Q}^{\ast })+\langle \nabla g(x_{Q}^{\ast }),x-x_{Q}^{\ast
}\rangle +\frac{m}{2}\left\Vert x-x_{Q}^{\ast }\right\Vert ^{2} \\
&\geq &g(x_{Q}^{\ast })+\frac{m}{2}\left\Vert x-x_{Q}^{\ast }\right\Vert
^{2}.
\end{eqnarray*}%
Using now the fact that for every $r\geq 1,a\geq 0,$ and $b\geq 0$ we have
\[
(a+b)^{r}\geq a^{r}+C_{r}b^{r}
\]%
with $C_{r}:=\inf_{t>0}\frac{(1+t)^{r}-1}{t^{r}}>0,$ we conclude that for
every $x\in Q$%
\[
h(x)\geq h_{Q}^{\ast }+C_{\frac{1}{2\theta }}\left( \frac{m}{2}\right) ^{%
\frac{1}{2\theta }}\left\Vert x-x_{Q}^{\ast }\right\Vert ^{\frac{1}{\theta }%
},
\]%
which completes the proof.
\end{proof}
Let us now prove the main result Theorem \ref{Th3}.
\begin{proof}
Firstly, by combining (\ref{S2}) with the assumption $\lambda ^{\prime }\leq 0,$ we
get for almost every $t\geq 0$
\begin{equation}
\left\Vert x^{\prime }(t)\right\Vert ^{2}+\left( \lambda (t)\left(
f(x(t)-f_{Q}^{\ast }\right) \right) ^{\prime }\leq 0.  \label{Q2}
\end{equation}%
Let us introduce the function
\begin{align*}
\varphi (t)& :=\frac{1}{2}\left( dist(x(t),\arg \min_{Q}f)\right) ^{2} \\
& =\frac{1}{2}\left\Vert x(t)-\left[ x(t)\right] \right\Vert ^{2},
\end{align*}%
where $\left[ x(t)\right] =P_{\arg \min_{Q}f}(x(t))$ is the projection of $%
x(t)$ onto the convex and closed subset $\arg \min_{Q}f$ . It is well-known
that the function $\varphi $ is of class $C^{1}$ and satisfies
\[
\varphi ^{\prime }(t)=\langle x^{\prime }(t),x(t)-\left[ x(t)\right] \rangle
.
\]%
Hence by letting $w=\left[ x(t)\right] $ in (\ref{S1}), we obtain
\[
\varphi ^{\prime }(t)+\left\Vert x^{\prime }(t)\right\Vert ^{2}+\lambda
(t)\left( f(x(t)-f_{Q}^{\ast }\right) ^{\prime }+\lambda (t)\langle \nabla
f(x(t)),x(t)-\left[ x(t)\right] \rangle \leq 0.
\]%
Using now the assumption $\lambda ^{\prime }\leq 0$ and the convex inequality%
\begin{align*}
\langle \nabla f(x(t)),x(t)-\left[ x(t)\right] \rangle & \geq f(x(t))-f(%
\left[ x(t)\right] ) \\
& =f(x(t))-f_{Q}^{\ast },
\end{align*}%
we infer that for almost every $t\geq 0$%
\begin{equation}
\varphi ^{\prime }(t)+\left( \lambda (t)\left( f(x(t)-f_{Q}^{\ast }\right)
\right) ^{\prime }+\lambda (t)\left( f(x(t))-f_{Q}^{\ast }\right) \leq 0.
\label{Q3}
\end{equation}%
On the other hand, since the functions $f(x(t))-f_{Q}^{\ast }$ and $\lambda
(t)$ are bounded on $[0,+\infty \lbrack $ and $\theta \leq \frac{1}{2},$
there exists some constant $c_{1}>0$ independent of $t$ such that
\[
\left( f(x(t))-f_{Q}^{\ast }\right) \geq c_{1}\left( \lambda (t)\left(
f(x(t)-f_{Q}^{\ast }\right) \right) ^{\frac{1}{2\theta }}.
\]%
Combining this inequality with the fact that $f$ satisfies the (GHEB)
hypothesis (\ref{e7}), we infer that
\begin{align}
\lambda (t)\left( f(x(t))-f_{Q}^{\ast }\right) & =\frac{1}{2}\lambda
(t)\left( f(x(t))-f_{Q}^{\ast }\right) +\frac{1}{4}\lambda (t)\left(
f(x(t))-f_{Q}^{\ast }\right) +\frac{1}{4}\lambda (t)\left(
f(x(t))-f_{Q}^{\ast }\right)   \nonumber \\
& \geq \frac{1}{2}\lambda (t)\left( f(x(t))-f_{Q}^{\ast }\right) +\frac{%
\kappa ^{\frac{1}{\theta }}}{4}\lambda (t)\left( \varphi (t)\right) ^{\frac{1%
}{2\theta }}+\frac{c_{1}}{4}\lambda (t)\left( \alpha (t)\left(
f(x(t)-f_{Q}^{\ast }\right) \right) ^{\frac{1}{2\theta }}  \nonumber \\
& \geq \frac{1}{2}\lambda (t)\left( f(x(t))-f_{Q}^{\ast }\right)
+c_{2}\lambda (t)\left( \varphi (t)+\lambda (t)\left( f(x(t)-f_{Q}^{\ast
}\right) \right) ^{\frac{1}{2\theta }},  \nonumber \\
& =\frac{1}{2}\lambda (t)\left( f(x(t))-f_{Q}^{\ast }\right) +c_{2}\lambda
(t)(\psi (t))^{\frac{1}{2\theta }},  \label{Q4}
\end{align}%
where $c_{2}>0$ is absolute constant and%
\[
\psi (t):=\varphi (t)+\lambda (t)\left( f(x(t)-f_{Q}^{\ast }\right) .
\]%
Inserting (\ref{Q4}) in the inequality (\ref{Q3}), we obtain for almost
every $t\geq 0$
\begin{equation}
\psi ^{\prime }(t)+c_{2}\lambda (t)\left( \psi (t)\right) ^{\frac{1}{2\theta
}}+\frac{1}{2}\lambda (t)\left( f(x(t)-f_{Q}^{\ast }\right) \leq 0.
\label{Q5}
\end{equation}%
In particular, we have the differential inequality%
\[
\psi ^{\prime }(t)+c_{2}\lambda (t)\left( \psi (t)\right) ^{\frac{1}{2\theta
}}\leq 0,
\]%
which implies in particular that the function $\psi $ is non increasing.
Since $\psi $ is non negative, then if there exists $t_{0}\geq 0$ such that $%
\psi (t_{0})=0$ then $\psi (t)=0$ for every $t>t_{0}$. Hence, in order to
estimate the growth of $\psi (t)$ for $t$ large enough, we can assume that $%
\psi (t)>0$ for all $t>0$. Therefore, by dividing the previous
differential inequality by $\psi ^{\frac{1}{2\theta }}(t)$, integrating the
resulting inequality and using the fact that $\Gamma (t)\rightarrow +\infty $
as $t\rightarrow +\infty ,$ we obtain
\begin{align}
\psi (t)& =O(e^{-c_{2}\Gamma (t)})\text{ if }\theta =\frac{1}{2},  \label{Q6}
\\
\psi (t)& =O(\left( \Gamma (t)\right) ^{-\frac{2\theta }{1-2\theta }})\text{
if }\theta \in ]0,\frac{1}{2}[.  \label{Q7}
\end{align}%
In particular, $\psi (t)\rightarrow 0$ as $t\rightarrow +\infty ;$ hence, by
integrating (\ref{Q5}) on the interval $[t,+\infty \lbrack $, we get%
\[
\int_{t}^{+\infty }\lambda (s)\left( f(x(s)-f_{Q}^{\ast }\right) ds\leq
2\psi (t).
\]%
Using now the fact that the function $\lambda (t)\left( f(x(t)-f_{Q}^{\ast
}\right) $ is decreasing, which is a consequence of (\ref{Q2}), we deduce
that for every $t>0$
\begin{align}
\frac{t}{2}\lambda (t)\left( f(x(t))-f_{Q}^{\ast }\right) & \leq \int_{\frac{%
t}{2}}^{t}\lambda (s)\left( f(x(s)-f_{Q}^{\ast }\right) ds  \nonumber \\
& \leq 2\psi (\frac{t}{2}).  \label{RR1}
\end{align}%
Combining (\ref{Q6}), (\ref{Q7}) and (\ref{RR1}) with the inequality
\begin{equation}
\Gamma (t)\leq 2\Gamma (\frac{t}{2}),~\forall t\geq 0,  \label{Ess}
\end{equation}%
which is a consequence of the fact that the function $\lambda $ is
deceasing, we deduce that%
\begin{align}
f(x(t))-f_{Q}^{\ast }& =O(\frac{1}{t\lambda (t)}e^{-M\Gamma (t)})\text{
if }\theta =\frac{1}{2},  \label{Q8} \\
f(x(t))-f_{Q}^{\ast }& =O(\frac{1}{t\lambda (t)}\left( \Gamma (t)\right) ^{-%
\frac{2\theta }{1-2\theta }})\text{ if }\theta \in ]0,\frac{1}{2}[,
\label{Q9}
\end{align}%
where $M>0$ is an absolute constant. This completes the proof of (\ref{H4}) and
(\ref{N2}). Let us now focus our attention on the study of the strong convergence of
$x(t)$ as $t\rightarrow +\infty .$ Integrating (\ref{Q2}) between $t$ and $2t
$ and using the Cauchy-Schawrz inequality, we obtain%
\[
\int_{t}^{2t}\left\Vert x^{\prime }(s)\right\Vert ds\leq \sqrt{t\lambda
(t)\left( f(x(t))-f_{Q}^{\ast }\right) }.
\]%
Dividing this inequality by $t$ and integrating the resulting differential inequality on $[\frac{\tau }{2},\infty
\lbrack $ where $\tau >0,$ we get, thanks to Fubini's theorem, the following
inequality%
\[
\ln 2~\int_{\tau }^{+\infty }\left\Vert x^{\prime }(s)\right\Vert ds\leq
\int_{\frac{\tau }{2}}^{+\infty }\sqrt{\frac{\lambda (s)}{s}\left(
f(x(s))-f_{Q}^{\ast }\right) }ds.
\]%
Using now the estimates (\ref{Q8}) and (\ref{Q9}), we deduce that
\begin{align}
\int_{\tau }^{+\infty }\left\Vert x^{\prime }(s)\right\Vert ds& =O\left(
\int_{\frac{\tau }{2}}^{+\infty }\frac{e^{-\frac{M}{2}\Gamma (s)}}{s}%
ds\right) \text{ }  \nonumber \\
& =O\left( e^{-\frac{M}{4}\Gamma (\frac{\tau }{2})}\right) \text{ if }%
\theta =\frac{1}{2},  \label{C1} \\
\int_{\tau }^{+\infty }\left\Vert x^{\prime }(s)\right\Vert ds& =O\left(
\int_{\frac{\tau }{2}}^{+\infty }\frac{\left( \Gamma (s)\right) ^{-\frac{%
\theta }{1-2\theta }}}{s}ds\right) \text{ if }\theta \in ]0,\frac{1}{2}[.
\label{C2}
\end{align}%
Therefore $x(t)$ converges strongly as $t\rightarrow +\infty $ to some $%
x^{\ast }$ which is, from Lemma \ref{Le1} and Theorem \ref{Th1}, an element of $\arg \min_{Q}f.$
Finally, the estimates (\ref{H3})and(\ref{N1}) on the decay rate of $x(t)$ can be easily
deduced from (\ref{C1}) and (\ref{C2}) by using a simple change of variable,
the estimate (\ref{Ess}) and the fact that
\[
\left\Vert x(\tau )-x^{\ast }\right\Vert \leq \int_{\tau }^{+\infty
}\left\Vert x^{\prime }(s)\right\Vert ds.
\]
\end{proof}
\par The proof of Corollary \ref{Co1} is a direct application of Theorem \ref{Th3} and the following elemetary result.
\begin{lemma}
Let $u,v:[0,+\infty[\rightarrow[0,+\infty[$ be two locally integrable functions such that $u(t)\sim Ct^{\alpha}$ and$v(t)\sim Ct^{\beta}$ as $t\rightarrow +\infty$ for some constants $C>0$ and $\alpha>-1$ and$\beta<-1$. Then $\int_{0}^{t} u(s)ds\sim \frac{C}{1+\alpha} t^{1+\alpha}$ and $\int_{t}^{+\infty} v(s)ds\sim -\frac{C}{1+\beta} t^{1+\beta}$ as $t\rightarrow +\infty$.
\end{lemma}
\section{On the strong convergence of the gradient system (GS) under a local Lojasiewicz inequality}
In this section we prove Theorem \ref{Th4}. The proof is inspired by the book of Haraux and Jendoubi \cite{HJ} and the paper of Chill and Fioranza \cite{CF}. As we will see in the proof, the hypothesis on the convexity of the objective function $f$ is not necessary in Theorem \ref{Th4} and the assumption $x_{0}^{\ast}$ is a global minimizer of $f$ can be replaced by the weaker one $x_{0}^{\ast}$ is  a minimizer of $f$ over the ball $B_{\mathcal{H}%
}(x_{0}^{\ast },r^{\ast})$. Moreover, we will notice that if $r^{\ast}=+\infty$ then $r$ can be taken equal to $+\infty$ too.
\begin{proof}
Let $r \in ]0,r ^{\ast }]$ to be chosen later and let $x_{0}\in B_{%
\mathcal{H}}(x_{0}^{\ast },r )$. Let $x\in C^{1}([0,+\infty \lbrack ,\mathcal{H})$ be
the unique global solution of the system (GS). We first notice that from the
assumption on the function $\lambda ,$ the function $\Gamma :[0,+\infty
\lbrack \rightarrow \lbrack 0,+\infty \lbrack ,$ defined by (\ref{e6}), is a
bijection of class $C^{1}.$ Let $y$ be the function defined on $[0,+\infty
\lbrack $ by $y(t)=x(\Gamma ^{-1}(t)).$ Using the chain rule, we easily
verify that $y$ belongs to the space $C^{1}([0,+\infty \lbrack ,\mathcal{H})$
and it is the unique solution of the system%
\begin{equation}
\left\{
\begin{array}{c}
y^{\prime }(t)=-\nabla f(y(t)),~t\geq 0 \\
y(0)=x_{0}.%
\end{array}%
\right.   \tag{NGS}
\end{equation}%
We distinguish two cases:
\par\noindent\underline{The first case:} Assume that there exists $t_{0}\geq 0$ such that
$f(y(t_{0}))=f^{\ast }.$ Then $\nabla f(y(t_{0}))=0.$ Therefore, from the
Cauchy Lipschitz theorem we deduce that for every $t\geq
0,~y(t)=y(t_{0})=x_{0}.$ Hence $f(x_{0})=f^{\ast }$ and $x(t)=x_{0}$ for
every $t\geq 0.$
\par\noindent\underline{The second case:} We assume here that for every $t\geq
0,f(y(t))>f^{\ast }.$ Let $t^{\ast }:=\sup \{t\geq 0:y(s)\in B_{\mathcal{H}%
}(x_{0}^{\ast },r^{\ast})$ for every $s\in \lbrack 0,t]\}$ and
define, on the interval on $[0,t^{\ast }[,$ the function $h(t)=\varphi
(f(y(t))-f^{\ast }).$ From the system (NGS) and the assumption (\ref{Ch1}), we have
for every $t\in \lbrack 0,t^{\ast }[$%
\begin{eqnarray}
-h^{\prime }(t) &=&-\varphi ^{\prime }(f(y(t)-f^{\ast })\langle \nabla
f(y(t)),y^{\prime }(t)\rangle   \nonumber \\
&=&\varphi ^{\prime }(f(y(t)-f^{\ast })\left\Vert \nabla f(y(t))\right\Vert
\left\Vert y^{\prime }(t)\right\Vert   \label{N1} \\
&\geq &\left\Vert y^{\prime }(t)\right\Vert .  \label{Naa}
\end{eqnarray}%
Therefore, for every $t\in \lbrack 0,t^{\ast }[$%
\begin{eqnarray*}
\left\Vert y(t)-x_{0}\right\Vert  &\leq &h(0)-h(t) \\
&\leq &\varphi (f(x_{0})-f(x_{0}^{\ast })),
\end{eqnarray*}%
Using now the continuity of $f$ and $\varphi $ and the fact that $\varphi
(0)=0,$ we infer that up to choose $r $ small we can assume that $t^{\ast
}=+\infty .$ Multiplying the equality (\ref{N1}) by $\varphi ^{\prime
}(f(y(t)-f^{\ast })$ and using the fact that $\left\Vert y^{\prime
}(t)\right\Vert =\left\Vert \nabla f(y(t))\right\Vert $ and the assumption (\ref{Ch1}), we get for every $t\geq 0$%
\[
-\varphi ^{\prime }(f(y(t)-f^{\ast })h^{\prime }(t)\geq 1.
\]%
Recalling that $\varphi (s)=\frac{s^{\theta }}{ \kappa\theta },$ we obtain
\begin{equation}
-\frac{1}{\kappa}\left( \kappa\theta \right) ^{\frac{\theta -1}{\theta }}\left(
h(t)\right) ^{\frac{\theta -1}{\theta }}h^{\prime }(t)\geq 1.  \label{Na}
\end{equation}%
Let us now suppose that $\theta \in ]\frac{1}{2},1].$ Integrating the last
differential inequality \ref{Na}), we infer that for every $t\geq 0$ we have
\begin{eqnarray*}
t &\leq &\frac{M_{\theta }}{2\theta -1}\left( \left( h(0)\right) ^{\frac{%
2\theta -1}{\theta }}-\left( h(t)\right) ^{\frac{2\theta -1}{\theta }%
}\right)  \\
&\leq &\frac{M_{\theta }}{2\theta -1}\left( h(0)\right) ^{\frac{2\theta -1}{%
\theta }},
\end{eqnarray*}%
where $M_{\theta }=\frac{\theta}{\kappa} \left( \kappa\theta \right) ^{\frac{\theta -1}{%
\theta }}>0,$ which is impossible, then in this case, $\theta $ must be in $%
]0,\frac{1}{2}].$ Now from a simple integration of (\ref{Na}), we
deduce that for every $t\geq 0,$
\[
h(t)\leq h(0)e^{-\frac{2}{\kappa}t},~\text{if }\theta =\frac{1}{2},
\]%
and%
\[
h(t)\leq \left( \left( h(0)\right) ^{\frac{2\theta -1}{\theta }}+\frac{%
1-2\theta }{M_{\theta }}t\right) ^{-\frac{\theta }{1-2\theta }},\text{ it }%
0<\theta <\frac{1}{2}
\]%
Therefore, by integrating the differential inequality \ref{Naa} between $t\geq 0$
and $+\infty ,$ we conclude that there exists $x_{\infty }$ in $\mathcal{H}$
such that for every $t\geq 0$ we have
\[
\left\Vert y(t)-x_{\infty }\right\Vert \leq h(0)e^{-\frac{2}{\kappa}t},~%
\text{if }\theta =\frac{1}{2},
\]%
and
\[
\left\Vert y(t)-x_{\infty }\right\Vert \leq \left( \left( h(0)\right) ^{%
\frac{2\theta -1}{\theta }}+\frac{1-2\theta }{M_{\theta }}t\right) ^{-\frac{%
\theta }{1-2\theta }},\text{ if }0<\theta <\frac{1}{2}.
\]%
Moreover, by using the fact that $h(t)=\frac{1}{\kappa\theta }\left(
f(y(t))-f^{\ast }\right) ^{\theta },$ we infer that the function $%
f(y(t))-f^{\ast }$ satisfies the following estimates for every $t\geq 0$%
\[
f(y(t))-f^{\ast }\leq \left( f(x_{0})-f^{\ast }\right) e^{-\frac{4}{\kappa%
}t}\text{ if }\theta =\frac{1}{2},
\]%
and
\[
f(y(t))-f^{\ast }\leq \left( \kappa\theta \right) ^{\frac{1}{\theta }%
}\left( \left( h(0)\right) ^{\frac{2\theta -1}{\theta }}+\frac{1-2\theta }{%
M_{\theta }}t\right) ^{-\frac{1}{1-2\theta }},\text{ if }0<\theta <\frac{1}{2%
}.
\]%
This ends the proof of Theorem \ref{Th4} since $x(t)=y(\Gamma (t)).$
\end{proof}

\end{document}